\documentclass[12pt]{article}
\usepackage{amsmath}
\usepackage{amsfonts}
\usepackage{amsthm}
\usepackage{graphicx}
\graphicspath{{Figures/}} 
\usepackage{overpic}
\usepackage{amssymb} 
\usepackage{pictex}
\usepackage{rotating}  
\usepackage{color}
\usepackage{cite} 
\usepackage{xcolor}

\usepackage{enumitem} 

\usepackage{accents} 

\usepackage{yhmath} 
\usepackage{etex} 
\usepackage{comment}

\numberwithin{equation}{section}

\newtheorem{theorem}{Theorem}[section]
\newtheorem{proposition}[theorem]{Proposition}
\newtheorem{lemma}[theorem]{Lemma}

\newtheorem{Definition}[theorem]{Definition}

\newtheorem{Remark}[theorem]{Remark}
\newenvironment{remark}{\begin{Remark}\rm}{\end{Remark}}
\newtheorem{RHproblem}[theorem]{RH problem}

\newtheorem{Example}[theorem]{Example}

\usepackage{color}

\newcommand{\C}{\mathbb{C}}
\newcommand{\D}{\mathbb D}

\newcommand{\R}{\mathbb{R}}

\newcommand{\T}{\mathbb{T}}

 \renewcommand{\AA}{\mathcal A} 
 
\newcommand{\CC}{\mathcal C}
\newcommand{\DD}{\mathcal D}
\newcommand{\EE}{\mathcal E}

\newcommand{\II}{\mathcal I}

\newcommand{\LL}{\mathcal L}

\newcommand{\OO}{\mathcal O}

\renewcommand{\Re}{{\rm Re} \,}
\renewcommand{\Im}{{\rm Im} \,}

\renewcommand{\bar}{\overline}
\renewcommand{\tilde}{\widetilde}
\renewcommand{\hat}{\widehat}

\begin{document}
\title{Zeros of Faber polynomials for Joukowski airfoils}
\author{N. Levenberg and F. Wielonsky}

\maketitle 

\begin{center} {\bf {In memory of J. Ullman}} \end{center}

\begin{abstract} {Let $K$ be the closure of a bounded region in the complex plane with simply connected complement whose boundary is a piecewise analytic curve with at least one outward cusp.} The asymptotics of zeros of Faber polynomials for $K$ are not understood in this general setting. Joukowski airfoils provide a particular class of such sets. We determine the (unique) weak-* limit of the full sequence of normalized counting measures of the Faber polynomials for Joukowski airfoils; {it is never equal to the potential-theoretic equilibrium measure of $K$. This implies that many of these airfoils admit an electrostatic skeleton} and also explains an interesting class of examples of Ullman \cite{U2} related to Chebyshev quadrature. 
\end{abstract}
%
%
\section{Faber polynomials of a compact set $K$}\label{faber}
Let $K\subset \C$ be a compact set consisting of more than one point such that the unbounded component $\Omega$ of $\bar\C\setminus K$ is simply connected. Let $\Phi$ be the (unique) conformal map from $\Omega$ to $\bar\C\setminus\bar\D$, 
$$\Phi:\Omega\to \bar\C\setminus\bar\D,$$
where $\D$ denotes the open unit disk, such that 
$$\Phi(\infty)=\infty,\qquad\Phi'(\infty)>0.$$ 
We denote by $\Psi$ the inverse map of $\Phi$. Then
$$\Phi(z)=\frac{z}{c_{K}}+a_{0}+\frac{a_{1}}{z}+\cdots,\quad
\Psi(z)=c_{K}z+b_{0}+\frac{b_{1}}{z}+\cdots,\qquad z\to\infty,$$
where $c_{K}$ denotes the logarithmic capacity of $K$. The {\it Faber polynomials $\{F_n\}$ for $K$} can be defined as follows:
$$F_n(z)=\Phi(z)^{n}+\OO(1/z),\qquad z\to\infty.$$
Equivalently,
\begin{equation}\label{altfaber} 
\frac{\Psi'(w)}{\Psi(w)-z}=\sum_{n=0}^{\infty}\frac{F_n(z)}{w^{n+1}}.
\end{equation}

To see a natural connection with potential theory, note that $P_{n}(z)=c_{K}^{n}F_n(z)$ are monic polynomials of degree $n$. By Cauchy's formula,
\begin{equation}\label{Pn-Cauchy}
P_{n}(z)=\frac{c_{K}^{n}}{2i\pi}\int_{\gamma_{\epsilon}}\frac{\Phi(t)^{n}}{t-z}dt,\qquad
z\in K,
\end{equation}
where $\gamma_{\epsilon}=\Psi(C_{1+\epsilon})$ and $C_{1+\epsilon}$ is the circle of radius $1+\epsilon$ centered at the origin. It follows from (\ref{Pn-Cauchy}) that
$$\limsup_{n\to \infty}\|P_{n}\|_{K}^{1/n}\leq(1+\epsilon)c_{K},$$
and letting $\epsilon$ go to $0$ we obtain 
$$\limsup_{n\to \infty}\|P_{n}\|_{K}^{1/n}\leq c_{K}.$$
Any monic polynomial $p$ of degree $n$ satisfies $||p||_K\geq  c_K^n$ so that, in fact, $$\lim_{n\to \infty}\|P_{n}\|_{K}^{1/n}=c_{K}.$$ Thus the $P_{n}$ are {\it asymptotically extremal} polynomials for $K$. Let 
$$\mu_n := \frac{1}{n}\sum_{j=1}^n \delta_{z_j^{(n)}}$$ 
where $z_1^{(n)},...,z_n^{(n)}$ are the zeros of $F_n$. We call $\mu_n$ the {\it normalized counting measure} of $F_n$. It follows that any weak-* subsequential limit $\mu$ of $\{\mu_n\}$ has a balayage to $\partial K$ which is the equilibrium measure $\mu_K$ of $K$ (cf., \cite[Theorem III.4.7]{ST}).

Ullman \cite{U1} proved a general result about limit points of zeros of the sequence of Faber polynomials $\{F_n\}_{n=1}^{\infty}$ for $K$. Building on Ullman's work, Kuijlaars and Saff \cite{KS} proved the following more refined result:

\begin{theorem}[{\cite[Theorem 1.5]{KS}}] If the interior $K^o$ of $K$ is empty, then 
$$\mu_n \to \mu_K, \quad\text{weak-*, as }n\to\infty.$$ 
If $K^o$ is connected and either 
\begin{enumerate}
\item $\partial K$ is not a piecewise analytic curve; or 
\item $\partial K$ is a piecewise analytic curve that has a singularity other than an outward cusp,
\end{enumerate}
then there is a subsequence of $\{\mu_n\}$ which converges in the weak-* topology to $\mu_K$.

\end{theorem} 

Here by ``outward cusp'' at $z_0\in \partial K$ we mean the exterior angle at $z_0$ is $2\pi$.

Suppose that $K$ is the closure of a region bounded by a piecewise analytic curve
$L$ such that $\Psi$ 
has at least one singularity  on the unit circle $\T$. Mina-Diaz \cite{MD} studied behavior of the Faber polynomials when $L$ has no inner cusps (i.e., with exterior angle zero) but satisfying an extra condition when the singularities are only smooth corners (i.e., the exterior angle is $\pi$) and outer cusps. This extra condition is that the so-called Lehman expansion of $\Psi$ about at least one of the singularities contains logarithmic terms, see \cite[Assumption A.2]{MD} for details.
In particular, in his setting, there is always a subsequence of the normalized counting measures $\{\mu_n\}$ that converges in the weak-* topology to $\mu_K$. Indeed, the whole sequence $\{\mu_{n}\}$ converges to $\mu_{K}$ if $L$ is a Jordan curve. By different methods, this last assertion was also proven to be true if $L$ has an inner cusp, see \cite[Corollary 3.2]{SS}.

To the best of our knowledge, other than the $m-$cusped hypocycloid studied by He and Saff \cite{HS}, there are no known results on asymptotics of $\{\mu_n\}$ when the singularities of 
$\partial K$ are only outward cusps, none of which satisfies the extra condition in \cite{MD}. In this note, we analyze the very natural case of Joukowski airfoils (described in the next section) and we describe precisely the (unique) weak-* limit of the full sequence $\{\mu_n\}$ in the ``real'' setting (Section 3) and the ``complex'' setting (Section 4). In particular this limit measure is {\it never} equal to $\mu_K$ {and hence provides an electrostatic skeleton for $K$; see Remark \ref{ES}}. This also ``explains'' an interesting class of examples of Ullman \cite{U2} related to Chebyshev quadrature (Section 5).

\section{Joukowski and Faber: our set-up}

A natural way to construct regions bounded by a piecewise analytic curve with an outward cusp is to take a classical {\it Joukowski airfoil}. Mathematically, $\Psi: \{z:|z|>1\}\to \C \setminus K$ is the composition $\Psi = J\circ T$ where 
$$J(\zeta) = {1\over 2} \left(\zeta+\frac1\zeta\right)$$ is the Joukowski map and $\zeta=T(z)=az+b$ with $a,b\in \C$ chosen so that $-1$ lies in the interior of $K$ and $1$ lies on $\partial K$, and here we have an outward cusp (notice that $1$ and $-1$ are the points where the derivative of $\zeta+\zeta^{-1}$ vanishes). Thus $\Psi$ is a particular kind of rational exterior mapping function as studied by Liesen \cite{L} (who utilizes (\ref{altfaber})). Indeed, 
$$\Psi(z)= \frac{a^2z^2+2abz+b^2+1}{2az+2b}.$$
In this case, it is clear that the expansion of $\Psi$ near the singularity $1$ does not contain any logarithmic terms.
It will often be more convenient to write
\begin{equation}\label{def-zeta}\zeta=T(z)=az+b=Re^{i\theta}(z-1)+1.\end{equation}
Here $R>1$ and $\theta\in (-\pi/2,\pi/2)$ must be chosen so that the circle 
$$\{\zeta = T(e^{it}): 0\leq t \leq 2\pi\}$$
surrounds the point $-1$. Note that $T(1)=1$ so that $\Psi(1) =1$ and we do, indeed, have an outward cusp at $z=1$. It follows that $R\cos \theta >1$; i.e., $\Re (Re^{i\theta})>1$. Our Joukowski airfoil $K$ is symmetric with respect to the real axis if and only if $\theta =0$ (of course $R>1$); we will call this the {\it real case}. The relation between $a,b, R, \theta$ is 
\begin{equation}\label{def-ab} a=Re^{i\theta}, \quad b= 1-Re^{i\theta} \quad \hbox{where} \ \Re b <0. \end{equation}
The real case corresponds to $b<0$.

Returning to \cite{L}, Liesen defines ``shifted'' Faber polynomials $\hat F_n$ which, in our setting, are simply related to $F_n$ by an additive constant:
\begin{equation}\label{link-shift} \hat F_n(z):=F_n(z) +(-b/a)^n. \end{equation}
In his equation (19) he gives an explicit formula for $ \hat F_n$. We modify his notation slightly to write
\begin{equation}\label{exp-shift} 
\hat F_n(z)=2a^{-n}V(z)^{n/2}T_n\left(\frac{W(z)}{V(z)^{1/2}}\right) \end{equation}
where
\begin{equation}\label{VW} V(z)=b^2+1-2bz, \quad W(z)=z-b \end{equation} 
-- thus our $W$ is $a^2$ times that of Liesen while our $V$ is $a$ times his -- and $T_n$ is the classical Chebyshev polynomial of the first kind:
$$T_n(z)=\frac{1}{2}\left([z+\sqrt{z^2-1}]^n+ [z-\sqrt{z^2-1}]^n\right).$$
Since $T_n$ is even if $n$ is even and odd if $n$ is odd, (\ref{exp-shift}) is independent of the choice of the square root for $V(z)^{1/2}$. We adopt the convention that
\begin{equation}\label{choice-root} V(1)^{1/2} = 1-b. \end{equation}
Even more explicitly, this gives
\begin{equation}\label{use?} F_n(z)=\left({1\over a}\right)^n\left[ \left(z+(-b) + \sqrt {z^2-1}\right)^n  +\left(z+(-b) - \sqrt {z^2-1}\right)^n  -(-b)^n\right].\end{equation}
We study the asymptotics of $z_1^{(n)},...,z_n^{(n)}$, the zeros of $F_n$, and the corresponding normalized counting measures 
$$\mu_n := \frac{1}{n}\sum_{j=1}^n \delta_{z_j^{(n)}}.$$
 
For future use, we define
\begin{equation} \label{def-U} U(z):= \frac{W(z)}{V(z)^{1/2}}= \frac{z-b}{\sqrt{b^2+1-2bz}} \end{equation}
and 
\begin{equation} \label{def-c}
c:=\frac{1}{2}\left(b+\frac1b\right)
\end{equation}
so that $V(c)=0$. Note that $U$ is defined and holomorphic in the complex plane outside of a branch cut joining $c$ to infinity. From (\ref{exp-shift}), the zeros of the shifted Faber polynomial $\hat F_n$ other than $c$ must occur at points $z\in \C$ such that $U(z)\in [-1,1]$. Let 
\begin{equation}\label{arcA}
\mathcal A:=\left\{z\in \C: U(z)= \frac{z-b}{\sqrt{b^2+1-2bz}}\in [-1,1]\right\}.
\end{equation}

\begin{lemma}\label{arc} Depending on $b$, the set $\mathcal A$ in (\ref{arcA}) is a simple arc joining $1$ and $-1$ or the union of $[-1,1]$ and a circle. It contains the point $b$. The point $c$ in (\ref{def-c}) does not belong to $\mathcal A$ unless $b=-1$. We have $U'(1/b)=0$. Finally, $1/b\in \mathcal A$ if and only if $b\in (-\infty, -1]$. 
\end{lemma}

\begin{proof} We have $U(z):= {W(z)}/{V(z)^{1/2}}\in [-1,1]$ if and only if 
$$U^2(z)= \frac{W^2(z)}{V(z)}=\frac{(z-b)^2}{b^2+1-2bz} =\rho \in [0,1].$$
This gives a parameterization of the set $\mathcal A$: 
$$z=b(1-\rho)\pm \sqrt {\rho(1-b^2+b^2\rho)}, \quad 0\leq \rho \leq 1,$$
from which follows the assertions in the lemma. In particular, $z=b$ for $\rho =0$ and $z=\pm 1$ for $\rho = 1$. A direct calculation shows $U'(1/b)=0$. Using the parameterization, $1/b\in \mathcal A$ occurs if and only if $\rho = 1-1/b^2\in [0,1]$ so that $b\in (-\infty, -1]$. \end{proof} 

Qualitatively we have three cases to consider/describe:
\begin{enumerate}
\item {\sl Case $b\not\in (-\infty, -1]$}: One checks that $\mathcal A$ is a simple arc. In the special case $b\in (-1,0)$ we have 
$\mathcal A=[-1,1]$. Since $c\not\in \mathcal A$, a branch cut for $U$ can be taken to avoid $\mathcal A$. Moving along $\mathcal A$ from $-1$ to $1$, $U$ increases with $U(-1)=-1$; $U(b)=0$; and $U(1)=1$ (recall (\ref{choice-root})). In other words, giving $\mathcal A$ the positive orientation from $-1$ to $1$, $U:\mathcal A \to [-1,1]$ is a one-to-one, onto, increasing map.   
\item {\sl Case $b \in (-\infty, -1]$}: Define the circle
\begin{equation}\label{ctilde} 
\tilde {\mathcal C}_b:=\{z\in \C: |z-c|=c-b\}. \end{equation}
Note that $b,1/b$ are the points of intersection of $\tilde {\mathcal C}_b$ with the real axis; $\rho =0$ corresponds to $b$ while $\rho = 1-1/b^2$ corresponds to $1/b$. In this case $\mathcal A =[-1,1]\cup \tilde {\mathcal C}_b$ and $1/b\in [-1,1]\cap \tilde {\mathcal C}_b$; moreover the point $c$ lies inside $\tilde {\mathcal C}_b$ and hence any branch cut for $U$ intersects $\tilde {\mathcal C}_b$. For simplicity we take $(-\infty,c)$ as the branch cut. Since $U(b)=0$, $z\to U(z)$ is continuous as $z$ crosses $(-\infty,c)$. Note in this case $U(-1)=1$. Now as $z$ moves to the right along $[-1,1]$ starting at $-1$, $U(z)$ decreases from $U(-1)=1$ to $U(1/b)= -\sqrt{b^2-1}/b>0$. This is the minimum value $U$ attains on $[-1,1]$. Continuing, $U$ increases on $[1/b,1]$ as we move to the right from $1/b$ to $1$ where $U(1)=1$.  In particular, 
$U:[1/b,1] \to [-\sqrt{b^2-1}/b,1]$ is a one-to-one, onto, increasing map. One checks that for $z$ on the circle $\tilde {\mathcal C}_b$, the values of $U(z)$ vary continuously between $U(1/b)= -\sqrt{b^2-1}/b$ and $U(b)=0$. 
\item {\sl Case $b =-1$}: In this special case of the previous one,  $\mathcal A =[-1,1]$ as $c=b=-1=\tilde {\mathcal C}_b$. Here  $U(-1)=0$ and $U(z)$ takes values from $0=U(-1)$ to $1=U(1)$ as $z$ moves from $-1$ to $1$ along $\mathcal A =[-1,1]$. 

\end{enumerate}
The behavior of $U(z)$ on $[-1,1]$ when $b$ is real (and negative) is depicted in Figure~\ref{foncU}.
\begin{figure}[!tb]
\centering
\includegraphics[width=5cm]{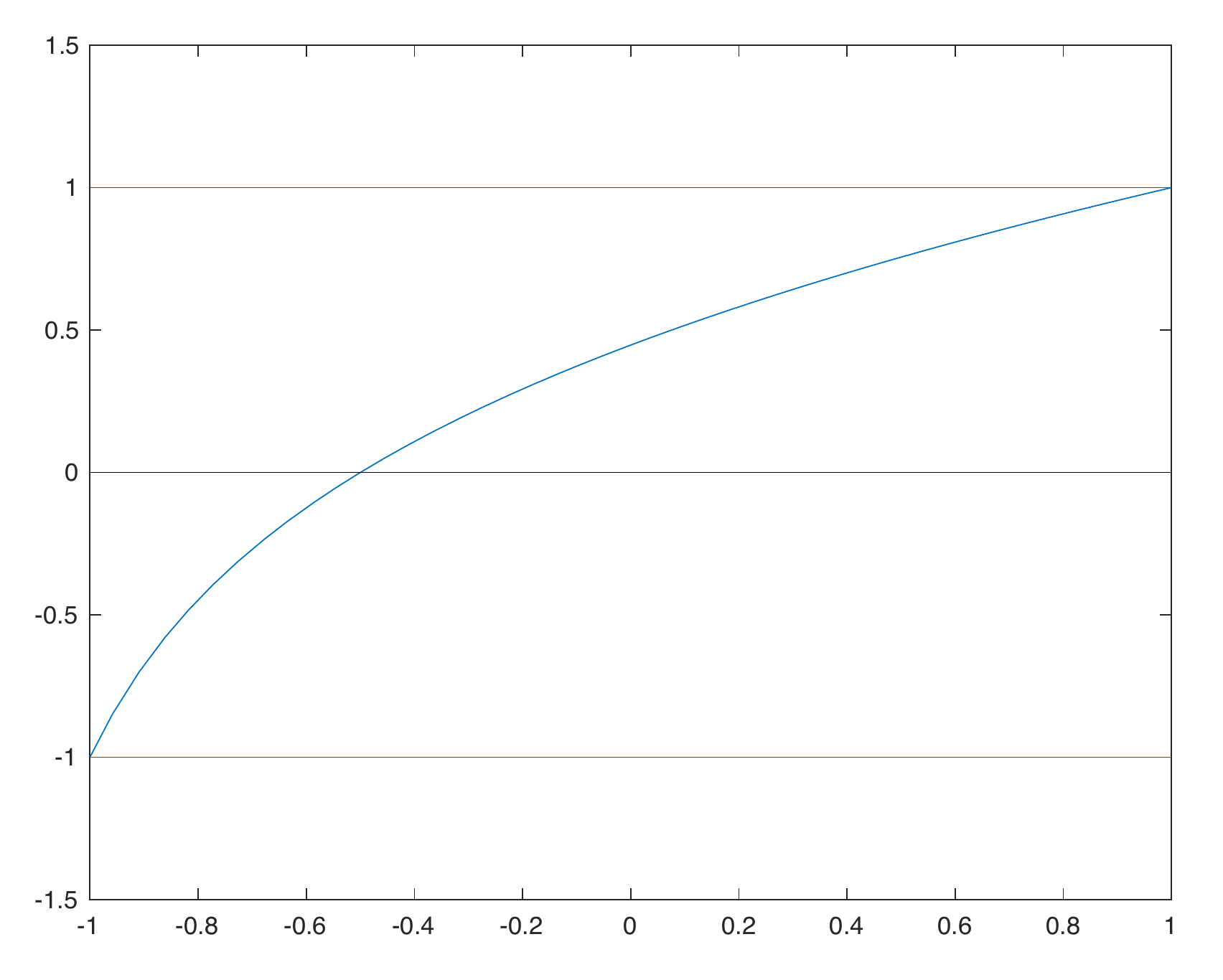}
\includegraphics[width=5cm]{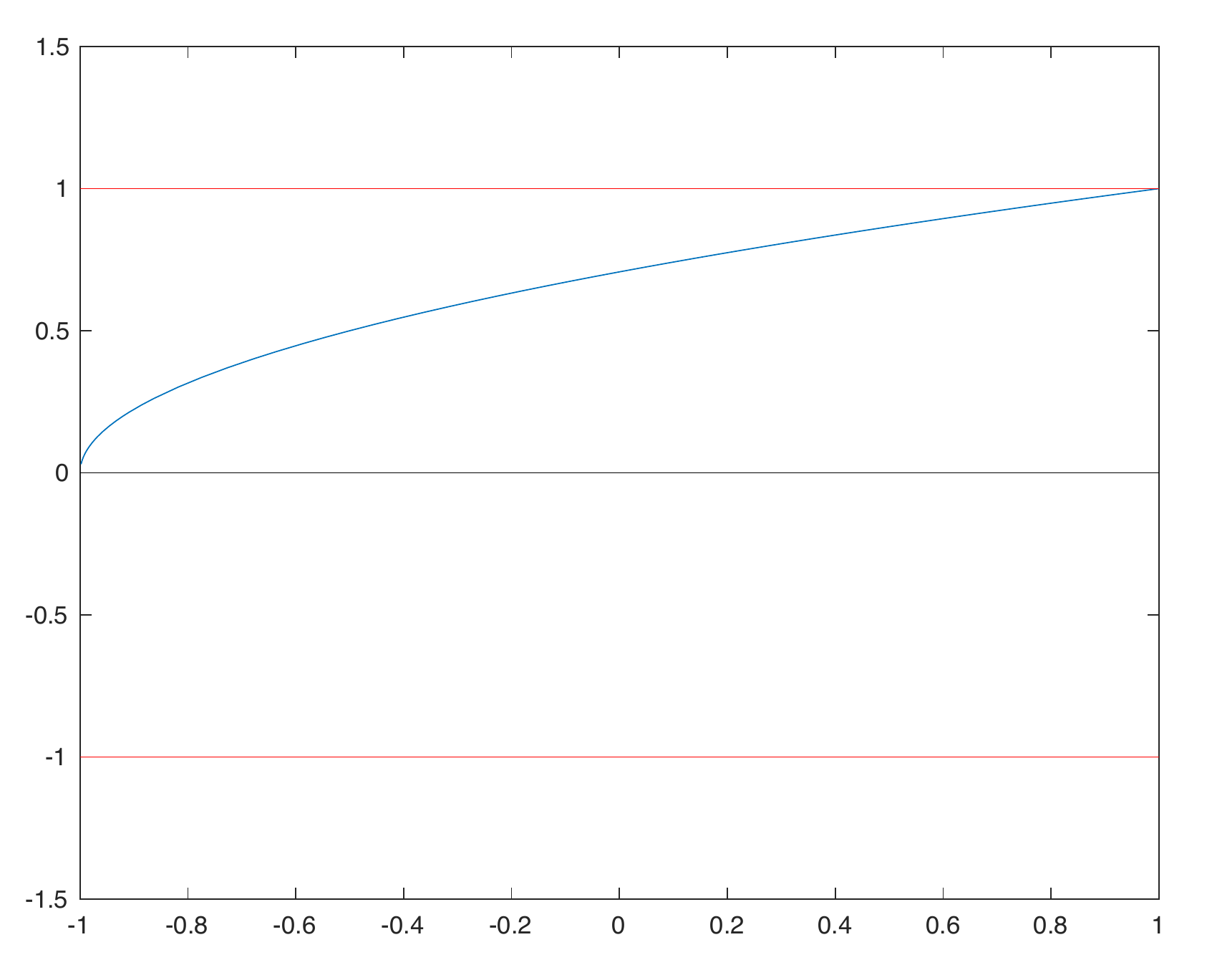}
\includegraphics[width=5cm]{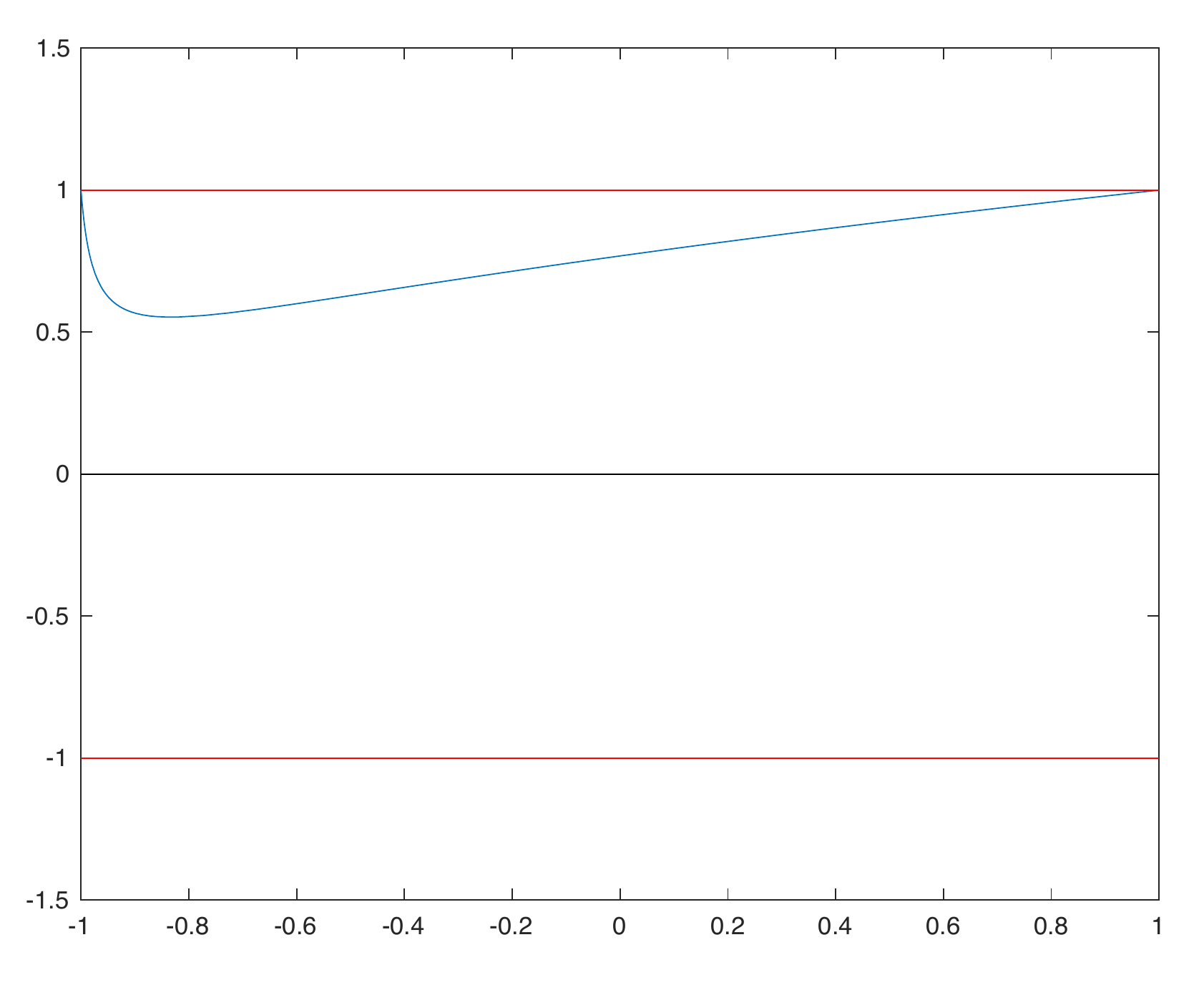}
\caption{Plot of $U(z)$, $z\in[-1,1]$, when $b$ is real: $b=-0.5$ (left), $b=-1$ (middle), $b=-1.2$ (right).}
\label{foncU}
\end{figure}

Our discussion of the asymptotics of the zeros of $F_n$, the Faber polynomials themselves, will involve the set $\mathcal A$ which is associated to the zeros of the shifted Faber polynomials $\hat F_n$. We separate into the {\it real} case ($\theta =0$) and the {\it complex} case ($\theta \not =0$) but a common ingredient will involve the circle
\begin{equation}\label{def-Cb} 
\mathcal C_b:=\{z\in \C: |z-c|=|b|/2\}=\{z\in \C: |V(z)|=|b|^2\}.\end{equation}
The equality in (\ref{def-Cb}) follows from the definitions of $V(z)$ and $c$. 
From our equations (\ref{exp-shift}) and (\ref{link-shift}), $F_n(z)=0$ holds if and only if
\begin{equation}\label{eq-Faber}
2T_n\left(\frac{W(z)}{V(z)^{1/2}}\right)=\left(\frac{-b}{V(z)^{1/2}}\right)^n.
\end{equation}
We isolate a simple but important observation from (\ref{def-Cb}):
\begin{proposition} \label{rhsmall} We have 
$$\left|\left(\frac{-b}{V(z)^{1/2}}\right)^n\right|\leq 1$$
if and only if $z$ lies outside or on $\mathcal C_b$.
\end{proposition}

We will consider two subcases of our analysis of the asymptotics of the zeros of $F_n$ in each of the real and complex settings: whether or not the arc $\mathcal A$ and the circle $\mathcal C_b$ intersect. We next determine when this occurs.

\begin{lemma}\label{inter} The arc $\mathcal A$ and the circle $\mathcal C_b$ intersect if and only if $R\cos \theta \geq 3/2$. In this case, there is a single point of intersection
$$i_b:=b+\sqrt \rho be^{i\alpha}$$
where 
\begin{equation}\label{eq-rho} 
\rho =\frac{(b^2+\bar b^2-1)^2}{4|b|^4}\in [0,1]
\end{equation}
and where $x=e^{i\alpha}$ is the root of the equation
\begin{equation}\label{eq-phase}
x^2+2\sqrt \rho x +(1-1/b^2)=0
\end{equation}
of modulus one. This root is unique if $\rho\not = 0$. 

When  $R\cos \theta =3/2$ the point of intersection is $i_b= -1$ and when $b$ is real, $i_b=1/2b$. 
\end{lemma}

\begin{proof} The condition that $z\in \mathcal A \cap \mathcal C_b$ entails
\begin{equation} \label{intersect} U^2(z)= \frac{W^2(z)}{V(z)}=\rho\in [0,1] \quad \hbox{and} \ |V(z)|=|b|^2. \end{equation}
Clearly then $|W(z)|=\sqrt \rho |b|$; and using the definitions $V(z)=b^2+1-2bz, \ W(z)=z-b$ from (\ref{VW}), 
$$V(z)+2bW(z)+(b^2-1)=0.$$
Replacing $V(z)$ by $W^2(z)/\rho$, we seek $z$ satisfying 
$$W^2(z)+2\rho bW(z)+\rho(b^2-1)=0 \quad \hbox{and} \quad |W(z)|=\sqrt \rho |b|.$$
Writing $W(z)=\sqrt \rho be^{i\alpha}$, we require $x=e^{i\alpha}$ to satisfy the quadratic equation
$$ x^2+2\sqrt \rho x +(1-1/b^2)=: x^2+2\sqrt \rho x +d= 0$$
which is (\ref{eq-phase}).

Using (\ref{eq-phase}) we show that (\ref{intersect}) has at most one solution. First, if (\ref{eq-phase}) has a solution  $x=e^{i\alpha}$ of modulus one then $z=b +\sqrt \rho be^{i\alpha}$ satisfies (\ref{intersect}) since $W(z)=z-b$. If (\ref{eq-phase}) has two distinct solutions  $x_1$ and $x_2$ of modulus one, since $x_1+x_2=-2\sqrt \rho \in \R$ we have either $x_1=-x_2$ or $x_1=\bar x_2$. If $x_1=-x_2$ then $\rho =0$ so that $z=b$; then $V(b)=0$ which gives $b=0$ which is impossible. If $x_1=\bar x_2$, then the product $x_1x_2= 1 = d=1-1/b^2$ which is impossible. 

Next we claim that (\ref{eq-phase}) cannot have (conjugate) reciprocal solutions $x=\beta e^{i\alpha}$ and $1/\bar x= \beta^{-1} e^{-i\alpha}$ with $\beta \not = 1$. For the sum $x+1/\bar x$ has $\Im (x+1/\bar x)=(\beta - \beta^{-1})\sin \alpha$ which vanishes if and only if $\alpha =0$; this implies $x$ and $1/\bar x=1/x$ are real with $x\cdot 1/x=1=  d=1-1/b^2$ which is impossible. We conclude that (\ref{eq-phase}) has a root of modulus one if and only if the polynomial $x^2+2\sqrt \rho x +d$ and its reciprocal $\bar dx^2+2\sqrt \rho x +1$ share a common root; i.e., if the resultant of these polynomials vanishes. A calculation gives that the vanishing of the resultant is equivalent to
$$4\rho(1-2\Re d + |d|^2)=(1-|d|^2)^2.$$ 
Using $d=1-1/b^2$ and rewriting this in terms of $b$, we have
$$ \frac{4\rho}{|b|^4}=\frac{1}{|b|^8}(\bar b^2+b^2-1)^2; \ \hbox{i.e.}, \  4\rho |b|^4=(\bar b^2+b^2-1)^2$$
which is (\ref{eq-rho}).

Note that if $|b|$ is small, the center $c=\frac{1}{2}(b+1/b)$ of $\mathcal C_b$ has large modulus. On the other hand, when $|b|$ is small $U(z)$ is very close to the identity and $\mathcal A$ stays in a fixed bounded region. Thus $\mathcal A \cap \mathcal C_b = \emptyset$ for such $b$. We characterize the values of $b$ which correspond to the first time(s) when $|b|$ is sufficiently large so that these sets intersect at a point. When this happens, by continuity this first intersection point must be at an endpoint of $\mathcal A$; i.e., at $1$ or $-1$. Then $\rho = U^2(z)=1$ and (\ref{eq-rho}) becomes $(b\pm \bar b)^2=1$ which  gives $\Re(b)=\pm 1/2$. Since we require $\Re b <0$ we must have $\Re(b)=-1/2$; i.e., $R\cos \theta = 3/2$. Using $\rho =1$ in (\ref{eq-phase}) we get, apriori, the roots $1/b -1$ and $-1/b-1$. We require the root to have modulus one and $|1/b -1|=1$ 
implies $|1-b|=|b|$ which cannot occur if $\Re b <0$. Finally we arrive at the root $e^{i\alpha} = -1/b-1$ which gives the (first) intersection point at $z=b+be^{i\alpha}=b+b(-1/b-1)=-1$ as required. 

If $b$ is real, using (\ref{eq-rho}) gives $i_b=b+\sqrt \rho b=1/2b$.
\end{proof}

If $R\cos \theta \geq 3/2$, the mapping 
$$\phi_b(z):= \frac{b-z-\sqrt{z^2-1}}{b}=\frac{b-J^{-1}(z)}{b},$$
will be useful  in the next sections.
If $b\not\in (-\infty, -1]$, we take the simple arc $\mathcal A$ as a branch cut $C$ for the square root; for $b$ real we take $C=[-1,1]$. Giving $C$ a positive orientation from $-1$ to $1$, for $x \in C$ we write 
$(\phi_b)_+(x)$ and $(\phi_b)_-(x)$ for the limits of $\phi_b(z)$ as $z\to x$ from the two sides of $C$. Note that $\phi_b(z)\not =1$ since $z+\sqrt{z^2-1}\not = 0$; but there exist $z$ with $|\phi_b(z)|=1$ and these points will be of interest. 
Define the curve
\begin{equation}\label{lb} 
\LL_{b}^{+}:=\{z\in \C: |\phi_b(z)|=1\}=\{z\in \C: |b-z-\sqrt{z^2-1}|=|b|\}.
\end{equation}
This is a loop (closed curve) which is a portion of the curve
$$
\LL_{b}:=\LL_{b}^{+}\cup\LL_{b}^{-},\qquad \hbox{where} \ \LL_{b}^{-}:=\{z\in \C: |b-z+\sqrt{z^2-1}|=|b|\}.
$$
The curve $\LL_{b}$, along with other curves of interest, are depicted in Figure \ref{curveLb}.
\begin{figure}[htb]
\centering
\def\svgwidth{14cm}
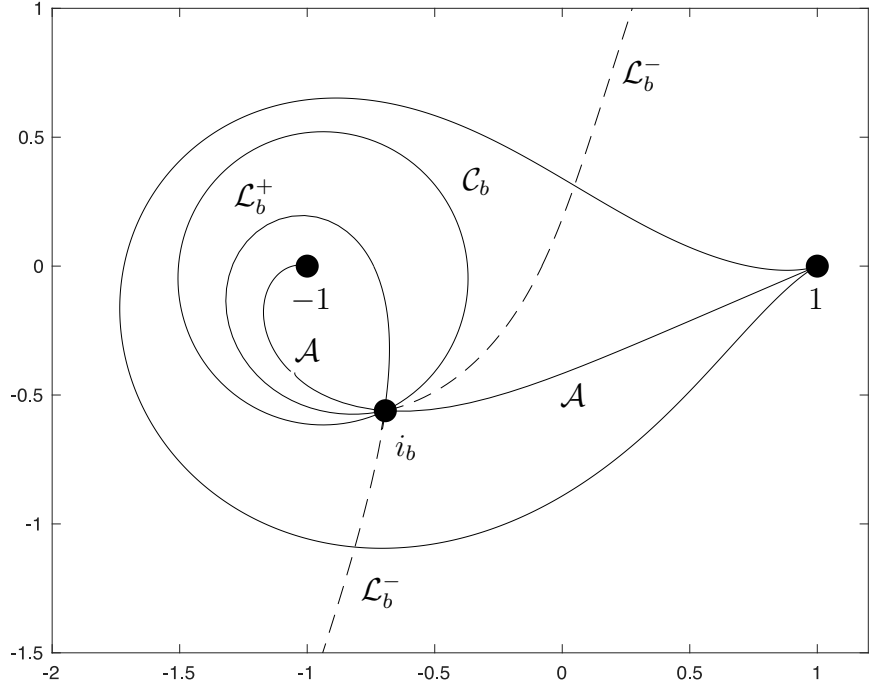
\caption{A Joukowski airfoil ($R=2.1$, $\theta=0.2$), along with the circle $\CC_{b}$, the arc $\AA$, the curve $\LL_{b}=\LL_{b}^{+}\cup\LL_{b}^{-}$. The loop $\LL_{b}^{+}$ lies inside the circle $\CC_{b}$; the remaining part $\LL_{b}^{-}$ of $\LL_{b}$ lies outside.}
\label{curveLb}
\end{figure}
We describe some of the properties of $\LL_{b}$ in the next lemma.
\begin{lemma}
The curve $\LL_{b}$ has a unique point of intersection with the circle $\mathcal C_b$, which is
 the point $i_b$ from Lemma \ref{inter}. The point $i_{b}$ is a double point of $\LL_{b}$, and it is also the unique point of intersection of $\LL_{b}$ with the curve $\AA$. The loop $\LL_{b}^{+}$ is the portion of $\LL_{b}$ which lies inside $\mathcal C_{b}$. When $b$ is real, $i_b = 1/2b$ and $\LL_b$ is symmetric about the real axis. 
\end{lemma}
\begin{proof} 
The preimage of $\LL_{b}$ under the Joukowski map 
$z=J(\zeta)$ is the circle $|b-\zeta|=|b|$, while the preimage of $\CC_{b}$ is the curve
$$
\left|\zeta-b+\left(\frac1\zeta-\frac1b\right)\right|=|\zeta-b|\left|1-\frac{1}{\zeta b}\right|=|b|.
$$
Hence, $\LL_{b}$ and $\CC_{b}$ intersect if and only if
$$|b-\zeta|=|b| \ \hbox{and} \ |\zeta b-1|=|\zeta b|\quad\text{i.e., }\quad |b-\zeta|=|b| \ \hbox{and} \ |\zeta-b^{-1}|=|\zeta|.$$
Since a circle and a line intersect at most twice, there are at most two solutions to the above system of equations, which are easily seen to be reciprocals of each other.
Thus, these two solutions are mapped by $J$ to the same point, the unique point of intersection of $\LL_{b}$ and $\CC_{b}$, which, moreover, has to be a double point of $\LL_{b}$. 
Furthermore, this point equals $i_{b}$. Indeed, on the one hand, $i_{b}\in\CC_{b}$. On the other hand, $i_{b}\in\AA$ which implies that $W^{2}(i_{b})=\rho V(i_{b})$, $\rho\in[0,1]$ and then
\begin{multline*}
|b-i_{b}\pm\sqrt{i_{b}^{2}-1}|=|W(i_{b})\mp\sqrt{W^{2}(i_{b})-V(i_{b})}|
\\
=\left|\frac{W(i_{b})}{\sqrt{V(i_{b})}}\mp\sqrt{\frac{W^{2}(i_{b})}{V(i_{b})}-1}\right|
\left|\sqrt{V(i_{b})}\right|=\left|\pm\sqrt{\rho}\mp\sqrt{\rho-1}\right|\left|\sqrt{V(i_{b})}\right|=1\cdot|b|=|b|,
\end{multline*}
so that $i_{b}\in\LL_{b}$. The above computation also shows that a point of intersection of $\AA$ and $\LL_{b}$ must lie on $\CC_{b}$ and hence coincide with $i_{b}$.

Now, recalling that $\AA$ (or its subarc $[-1,1]$ when $b\in(-\infty,1]$) was chosen as the branch cut in the definition of $\phi_{b}$, one concludes that $\LL_{b}^{+}$ is the portion of $\LL_{b}$ that either lies entirely inside or entirely outside of $\CC_{b}$. Since the point at infinity belongs to $\LL_{b}^{-}$, one concludes that $\LL_{b}^{+}$ is the portion of $\LL_{b}$ that lies entirely inside $\CC_{b}$.

For $b$ real, the symmetry of $\LL_b$ about the real axis is clear.
\end{proof}

To describe the asymptotics of the normalized counting measures $\{\mu_n\}$ of the zeros of $\{F_n\}$, we will need the equilibrium measure of the unit circle, $\T$:
$$\eta:= \mu_T=\frac{1}{2\pi}d\theta$$
and the equilibrium measure of the interval $[-1,1]$:
$${\mu_{[-1,1]}= \frac{1}{\pi}\frac{dx}{\sqrt{1-x^2}}.}$$
We recall that the normalized counting measures of the Chebyshev polynomials $\{T_n\}$ converge weak-* to $\mu_{[-1,1]}$.

Finaly, given a measurable map $f:A\to B$ between two measure spaces and a measure $\nu$ on $A$, we write $f_*(\nu)$ for the push-forward measure of $\nu$ under $f$.

\section{Zero distribution: the real case}

In this section, we assume $\theta=0$, i.e., $b<0$; the real case. The zero distribution of some Faber polynomials in this case are shown in Figure \ref{zeros-sym}.
 \begin{figure}[!tb]
 \centering
\vspace{-1cm}
\includegraphics[width=7.5cm]{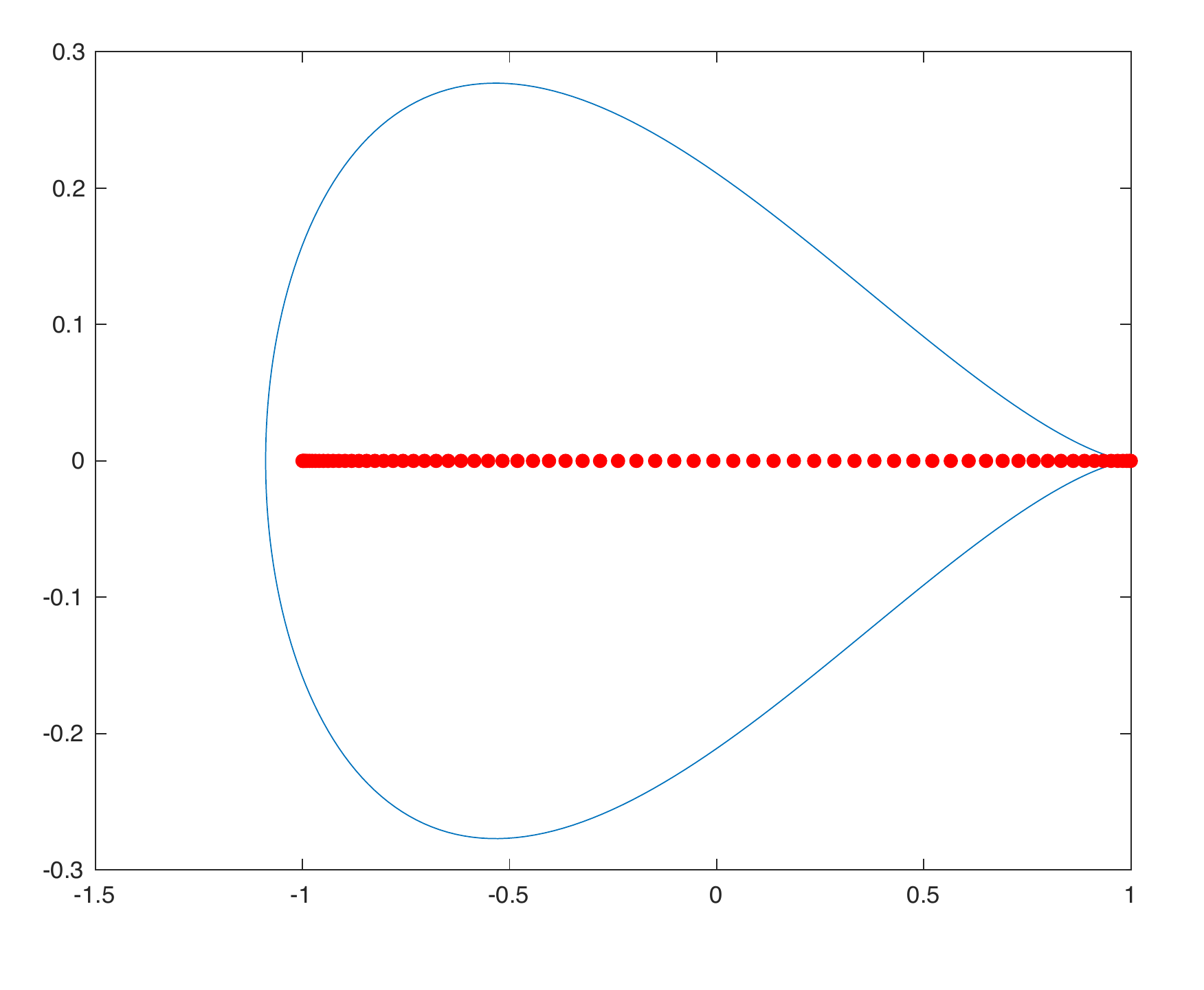}
\includegraphics[width=7.5cm]{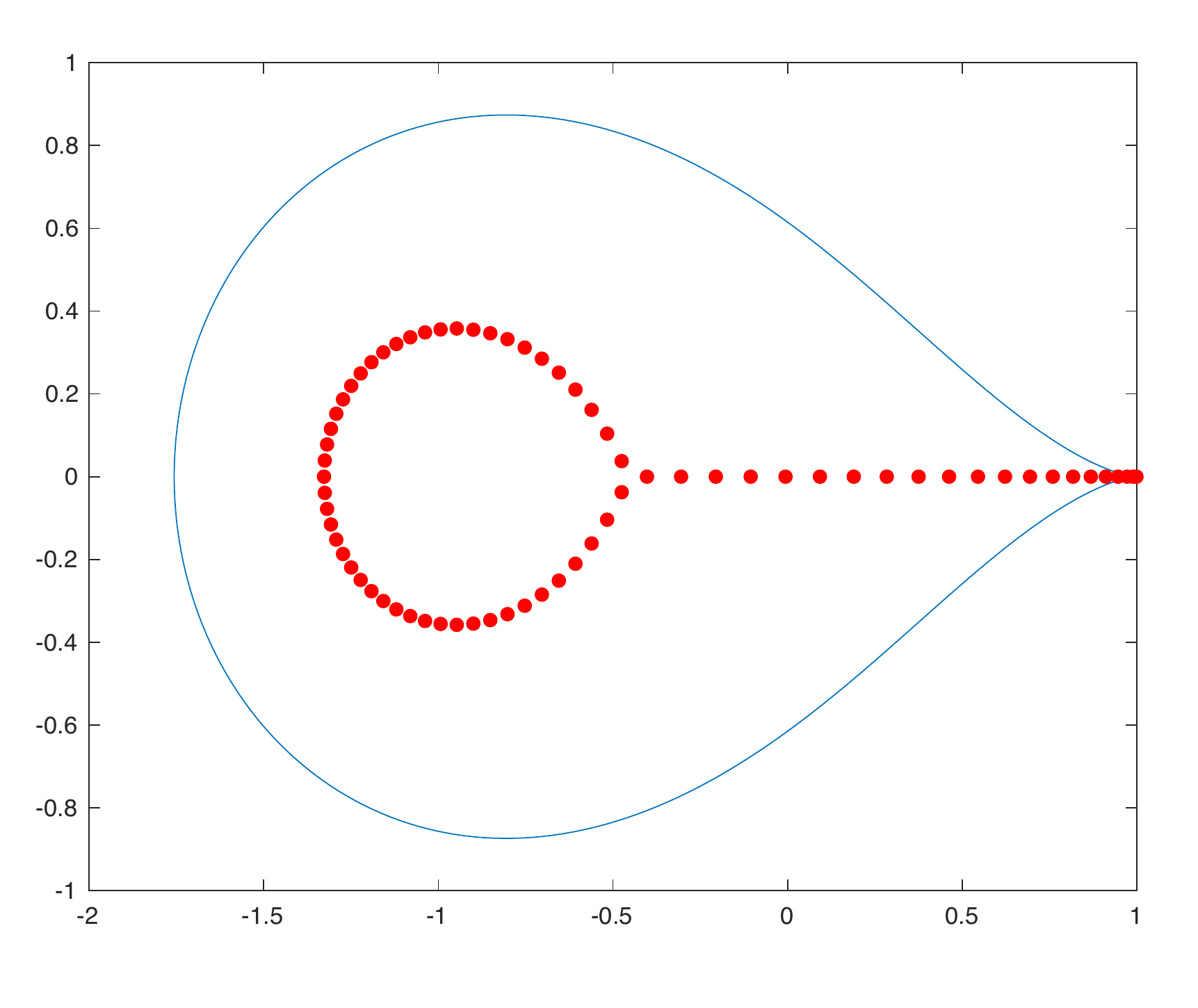}
\caption{Zero distribution of the Faber polynomials $F_{n}(z)$ in the real case. The degree $n=70$, and $R=1.26$ (left), $R=2.1$ (right).}
\label{zeros-sym}
\end{figure}

From Lemma \ref{inter}, we distinguish two subcases: $1<R \leq 3/2$ and $R>3/2$. 

\begin{theorem}\label{realullman} For $1<R \leq 3/2$, all zeros of $F_n(z)$ lie in $[-1,1]$ and 
$$\lim_{n\to \infty} \mu_n = (U^{-1})_*(\mu_{[-1,1]}) \quad \hbox{weak-*},$$
where the push-forward measure of $\mu_{[-1,1]}$ by $U^{-1}$ admits the following explicit expression:
\begin{equation}\label{push-expl}
(U^{-1})_*(\mu_{[-1,1]})=\frac1\pi\frac{1}{\sqrt{1-x^{2}}}\bigl(\frac{1-bx}{1+b^{2}-2bx}\bigr)dx,\qquad x\in(-1,1).
\end{equation}
\end{theorem}
\begin{proof} In this subcase, $0> b=1-R \geq -1/2$ so that $\mathcal A=[-1,1]$ and $U:[-1,1] \to [-1,1]$ is a one-to-one, onto, increasing map. Moreover, since $R \leq 3/2$, from Lemma \ref{inter}, $[-1,1]\cap \mathcal C_b$ is empty or consists of the point $-1$ so that, using Proposition \ref{rhsmall}, 
$$\left|\frac{b}{V(z)^{1/2}}\right|\leq 1 \quad \hbox{for} \ z\in [-1,1].$$
We adapt the argument of \cite[p. 422]{U2}, (see also Section \ref{quadrat} on Chebyshev quadrature below). The values of the Chebyshev polynomial $T_n(x)$ for $x\in [-1,1]$ oscillate between $-1$ and $1$, taking these values $n$ times each, at $x=\cos t, \ t={2k\pi}/{n}, \ k=0,1,...,n-1$ for the value $1$ and at $x=\cos t, \ t={(2k+1)\pi}/{n}, \ k=0,1,...,n-1$, for the value $-1$. It follows that between each $n$ pairs of oscillations, i.e., between $\cos ({2k\pi}/{n})$ and $\cos ({(2k+1)\pi}/{n})$, for each $z\in [-1,1]$ there is at least one value of $x$ so that $2T_n(x)=\left({-b}/{V(z)^{1/2}}\right)^n$ (as well as a zero of $T_n(x)$). Recalling from (\ref{eq-Faber}) that $F_n(z)=0$ if and only if 
$$2T_n(U(z))=\left(\frac{-b}{V(z)^{1/2}}\right)^n$$
and using the fact that $U:[-1,1] \to [-1,1]$ is monotone, we get exactly $n$ distinct solutions $z_1^{(n)},\ldots,z_n^{(n)}\in [-1,1]$ of this last equation; i.e., $n$ distinct zeros of $F_n$. Moreover, by the monotonicity of $U$ on $[-1,1]$ and the weak-* convergence of the normalized zero measures of the Chebyshev polynomials $\{T_n\}$  to $\mu_{[-1,1]}$, we conclude that $\mu_n$ converges weak-* to $(U^{-1})_*(\mu_{[-1,1]})$. Formula (\ref{push-expl}) comes from the fact that, by definition of the push-forward measure, the density of $(U^{-1})_*(\mu_{[-1,1]})$ with respect to $dx$ equals
$$
\frac1\pi\frac{U'(x)}{\sqrt{1-U^{2}(x)}},
$$
which is easily seen to be equal to the expression in the right-hand side of (\ref{push-expl}).
\end{proof}

We introduce some notation for the second case, $R>3/2$. Recall that 
$$\LL_{b}^{+}=\{z\in \C: |\phi_b(z)|=1\}=\{z\in \C: |b-z-\sqrt{z^2-1}|=|b|\}$$
is a loop which is symmetric about the real axis and contains the point $i_b = 1/2b$ where it has a corner. 
An example of such a loop is depicted in Figure \ref{loop} when $R=2.1$.
\begin{figure}[!tb]
 \centering
\vspace{-1cm}
\includegraphics[width=8cm]{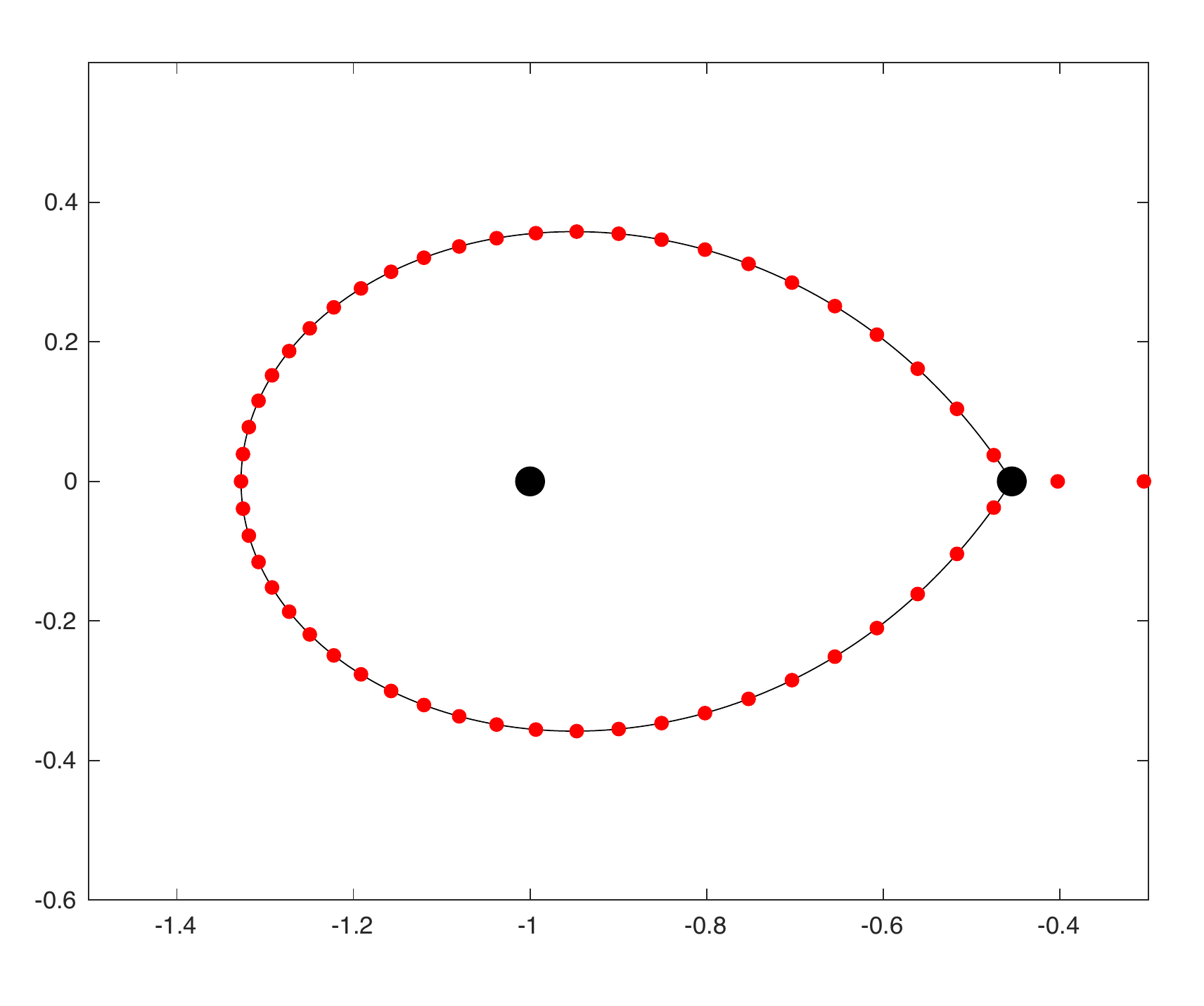}
\caption{Complex zeros of the Faber polynomials $F_{n}(z)$, $n=70$, $R=2.1$, accumulating on the loop (in black). The left dot is the real point $-1$, the right dot where the loop ends is the real point $1/2b<0$.}
\label{loop}
\end{figure}

Define
$$c_{\pm}:= (\phi_b)_{\pm}\left(\frac{1}{2b}\right) = 1-\frac{1}{2b^2} \mp 
\frac{i}{b}\sqrt {1-\frac{1}{4b^2}}\in \T$$
(note $1-1/4b^2>0$ since $b < -1/2$). The image $\phi_b(\LL_b^{+})$ is clearly a subarc of $\T$ from $c_+$ to $c_-=\bar c_+$, traversed counterclockwise (notice that $\phi_{b}(z)$ never takes the value $1$), symmetric about the real axis. We denote this arc by $(c_+,c_-)$. 
We also define the real segment 
$$I_b:=[1/2b,1].$$

\begin{theorem} \label{otherreal} For $R>3/2$, all zeros of $\{F_n\}$ accumulate on $\LL_b^{+}
 \cup I_b$. Moreover
\begin{equation}\label{lim-dist}
\lim_{n\to \infty} \mu_n = (U^{-1})_*(\mu_{[-1,1]})|_{I_b} + (\phi_b^{-1})_*(\eta|_{(c_+,c_-)} )\quad \hbox{weak-*}.
\end{equation}
\end{theorem}

\begin{proof} In this case, using $i_b= 1/2b$ and Proposition \ref{rhsmall}, for points $z\in [-1,1]$ we have 
$$\left|\frac{b}{V(z)^{1/2}}\right|\leq 1 \ \hbox{if and only if} \ z\in I_b=[1/2b,1].$$
Since $b < -1/2$ we have $-1< 1/2b$ so $I_b$ is a proper subinterval of $[-1,1]$. Recall that $U:[1/b,1] \to [-\sqrt{b^2-1}/b,1]$ is a one-to-one, onto, increasing map; hence $U$ is monotone on $I_b=[1/2b,1]\subset [1/b,1]$. As in the proof of Theorem \ref{realullman}, 
$$2T_n(U(z))=\left(\frac{-b}{V(z)^{1/2}}\right)^n$$
has real solutions $z$ for $z\in I_b$. Call these $z_1^{(n)},\ldots,z_{j(n)}^{(n)}\in I_b$ where $j(n)\leq n$ and define
$$\tilde \mu_n :=\frac{1}{n} \sum_{j=1}^{j(n)} \delta_{z_j^{(n)}}.$$
 Then, as in the previous result,   
$$\lim_{n\to \infty} \tilde \mu_n = (U^{-1})_*(\mu_{[-1,1]})|_{I_b} \quad\hbox{weak-*;}$$
i.e., these real roots distribute asymptotically like $(U^{-1})_*(\mu_{[-1,1]})|_{I_b}$. Note that the total mass of $(U^{-1})_*(\mu_{[-1,1]})|_{I_b}$ is
$$\mu_{[-1,1]}([U(1/2b),U(1)]) =\mu_{[-1,1]}([1-1/2b^2,1]) $$
$$=\frac{1}{\pi}\int_{1-1/2b^2}^1\frac{dx}{\sqrt{1-x^2}}= \frac{1}{\pi}\left(\frac{\pi}{2}-\sin^{-1}(1-1/2b^2)\right).$$

Next, we show that 
$$2T_n(U(z))=\left(\frac{-b}{V(z)^{1/2}}\right)^n$$
has no real solutions $z$ with $z\in [-1,1/2b]$ for $n$ sufficiently large. For such $z$, by Proposition \ref{rhsmall}, $|{b}/{V(z)^{1/2}}|> 1$. On the other hand, $U(z)$ takes only real values between $-\sqrt{b^2-1}/b$ and 1 for $z\in [-1,1/2b]$ so that $|2T_n(U(z))|\leq 2$. Thus for $n$ sufficiently large, $F_n$ has no zeros in $ [-1,1/2b]$.

Thus all other roots of $F_n$ lie outside of $[-1,1]$. We now show that there are no more roots on $\mathcal A=\{z\in \C: U(z)\in [-1,1]\}$ (which recall equals  $[-1,1]\cup \tilde {\mathcal C}_b$ if $R>2$ where $\tilde {\mathcal C}_b$ was defined in (\ref{ctilde})). Suppose $z\in \mathcal A$. We distinguish two cases as described following Lemma \ref{arc}. If $R\leq 2$ (i.e., $b\geq -1$), the two roots of $U^2(z)=x\in [0,1]$ lie in $[-1,1]$ and we are done by the previous paragraph. If $R>2$ (i.e., $b< -1$) and $U(z)\in [0,-\sqrt{b^2-1}/b)$ then $z \in \tilde {\mathcal C}_b\setminus [-1,1]$. Now $\tilde {\mathcal C}_b$ and ${\mathcal C}_b$ are concentric with ${\mathcal C}_b$ having a larger radius; thus by Proposition \ref{rhsmall}, $F_n$ has no roots on $\tilde {\mathcal C}_b$ and hence none on $\mathcal A$, other than those on $I_{b}$.

We conclude that all remaining roots of $F_n$ occur at points $z$ where $u:=U(z)\not\in [-1,1]$. We utilize the fact that the Chebyshev polynomials $T_n$ satisfy the asymptotic estimate 
$$T_n(u)= \frac12(u+\sqrt{u^2-1})^n\left(1+\OO\left(\frac{1}{\rho^{2n}}\right)\right),
$$
for $u$ outside of the ellipse $\EE_{\rho}$ given by $u=({w+w^{-1}})/{2}$ with $|w|=\rho>1$.
This follows from the definition
$$T_n(u)=\frac{1}{2}\left([u+\sqrt{u^2-1}]^n+ [u-\sqrt{u^2-1}]^n\right)=\frac{w^n+w^{-n}}{2}$$
where $u=({w+w^{-1}})/{2}$ and $u\not\in [-1,1]$ corresponds to $|w|>1$. Thus for $n$ large, roots $z$ of $F_n$ with 
$u=U(z)$ outside of $\EE_{\rho}$ satisfy
$$2T_n \left(\frac{W(z)}{V(z)^{1/2}}\right)= 
\left[ \frac{W(z)}{V(z)^{1/2}}+ \sqrt{\left(\frac{W(z)}{V(z)^{1/2}}\right)^2-1}\right]^n
\left(1+\OO\left(\frac{1}{\rho^{2n}}\right)\right)= \left(\frac{-b}{V(z)^{1/2}}\right)^n.
$$
We first consider the equation
\begin{equation}\label{approx-eq}
\left[ \frac{W(z)}{V(z)^{1/2}}+ \sqrt{\left(\frac{W(z)}{V(z)^{1/2}}\right)^2-1}\right]^n= \left(\frac{-b}{V(z)^{1/2}}\right)^n;
\end{equation}
i.e.,
$$\left(W(z)+\sqrt{(W(z))^2-V(z)}\right)^n = (-b)^n.$$
Recalling that $W(z)=z-b$ and $\phi_b(z)= ({b-z-\sqrt{z^2-1}})/{b}$, this gives the equation 
$(\phi_b(z))^n=1$. The solutions of this last equation are clearly the preimages under $\phi_{b}$ of the $n$-th roots of unity that lie on the arc $(c_+,c_-)$. 
Thus we see that, first of all, the set of accumulation points of the roots of (\ref{approx-eq}) is the entire curve $\LL_b^{+}$ from (\ref{lb}) whose image $\phi_b(\LL_b^{+})$ is the subarc $(c_+,c_-)$ of $\T$; moreover, the limit distribution of these roots is the push-forward under $\phi_b^{-1}$ of the uniform measure $\eta$ on $\T$ restricted to the arc $(c_+,c_-)$. 

Now, by the same computation as above, the roots of $F_{n}(z)$ with $U(z)$ outside of $\EE_{\rho}$ satisfy 
\begin{equation}\label{exact-eq}
(\phi_b(z))^n\left(1+\OO\left(\frac{1}{\rho^{2n}}\right)\right)=1.
\end{equation}
Hence, choosing $\rho>1$ as close as we wish to $1$, we see that all the roots of $F_{n}(z)$ accumulate on the loop $\LL_{b}^{+}$ as $n$ gets large. Making use of Rouch\'e's theorem, we next show that they have the same asymptotic distribution on $\LL_{b}^{+}$ as the roots of (\ref{approx-eq}). In order to control the magnitude of the $\OO$-term in (\ref{exact-eq}), we need to exclude from the subsequent analysis a neighborhood of $i_{b}=1/2b$, the unique point of $\LL_{b}^{+}$ whose image under $U$ belongs to $[-1,1]$. This neighborhood has to be small enough so that excluding from the analysis the zeros of $F_{n}$ belonging to that neighborhood does not modify the limit distribution, but it must also be large enough so that the $\OO$-term decreases sufficiently fast with $n$. We choose for this neighborhood a disk $ \DD_{b}$ centered at $i_{b}$ of radius $c/\sqrt{n}$ with $c>0$ chosen so that the image of $\LL_{b}^{+}\setminus \DD_{b}$ under $U$ lies outside of $\EE_{\rho}$ with $\rho=1+1/\sqrt{n}$ (an explicit value for $c$ could be given in terms of the derivative $U'(i_{b})\neq0$). Hence, outside of $ \DD_{b}$, the roots of $F_{n}$ satisfy (\ref{exact-eq}) with $\rho=1+1/\sqrt{n}$ (and the $\OO$-term is uniform with respect to $z$). 

Consider an $n$-th root of unity $a_{k}:=e^{2ik\pi/n}$ lying in $(c_{+},c_{-})\setminus\phi_{b}( \DD_{b})$ and a small circle $\CC_{k}$ of radius $n^{-2}$, centered at $a_{k}$, 
so that $\CC_{k}$ does not contain or encircle any other $n$-th roots of unity. 
To show that the contour $\Gamma_{k}:=\phi_{b}^{-1}(\CC_{k})$ surrounds exactly one root of $F_{n}$ for $n$ large enough, it is sufficent, by Rouch\'e's theorem and in view of (\ref{exact-eq}), to show that
$$
|(\phi_b(z))^n\OO(\rho^{-2n})|<|(\phi_b(z))^n-1|,\qquad z\in\Gamma_{k},
$$
or equivalently
\begin{equation}\label{Rouch}
\OO\left(\frac{1}{\rho^{2n}}\right)<\left|1-\left(a_{k}+\frac{e^{i\theta}}{n^{2}}\right)^{-n}\right|=
\left|1-\left(1+\frac{e^{i(\theta-2k\pi/n)}}{n^{2}}\right)^{-n}\right|
,\qquad\theta\in[0,2\pi].
\end{equation}
Since $\rho=1+1/\sqrt{n}$, we have that $\OO(\rho^{-2n})=\OO(e^{-2\sqrt{n}})$.
Moreover,
$$
\left(1+\frac{e^{i(\theta-2k\pi/n)}}{n^{2}}\right)^{-n}=1-\frac{e^{i(\theta-2k\pi/n)}}{{n}}+\OO\left(\frac{1}{n^{2}}\right).
$$
Consequently, the strict inequality (\ref{Rouch}) is satisfied for $n$ large enough, independent of $k$, showing that the contour $\Gamma_{k}$ surrounds exactly one root of $F_{n}$.

Finally, it remains to check that the non-real roots of $F_{n}$ excluded from the above argument do not modify the limit distribution given in (\ref{lim-dist}). Equivalently, we show that the number of roots of $F_{n}$ already found is asymptotically equivalent to $n$. 
First, notice that the total mass of $(\phi_b^{-1})_*(\eta|_{(c_+,c_-)} )$ is 
$$\eta \left((c_+,c_-)\right)= 2\cdot \frac{1}{2\pi}  \left(\pi - \cos^{-1} (1-1/2b^2)\right)=\frac{1}{\pi}  \left(\pi - \cos^{-1} (1-1/2b^2)\right),$$ 
while the total mass of $(U^{-1})_*(\mu_{[-1,1]})|_{I_b}$ is
$$\frac{1}{\pi}\left(\frac{\pi}{2} -\sin^{-1}(1-1/2b^2)\right).
$$
Next, the number of $n$-th roots of unity that are contained in the image of $\DD_{b}$ under $\phi_{b}$ is of order $\OO(\sqrt{n})$.
Hence an estimate for the number of roots of $F_{n}$ already found is
\begin{multline*}
\frac{n}{\pi}\left(\frac{3\pi}{2} -[ \sin^{-1}(1-1/2b^2)+\cos^{-1} (1-1/2b^2)]\right)-\OO(\sqrt{n})+o(n)
\\
=\frac{n}{\pi}\left(\frac{3\pi}{2} -\frac{\pi}{2}\right)+o(n)=n+o(n),
\end{multline*}
which is indeed asymptotically equivalent to $n$.
\end{proof}

\section{Zero distribution: the complex case}

In this section we take $\theta\not = 0$ (with $R\cos \theta >1$). 
To describe the limit distribution of the zeros of the Faber polynomials, we essentially repeat the analysis performed in the real case ($\theta=0$), with some modifications.

We recall that the arc $\AA$ and the circle $\CC_{b}$ were defined in (\ref{arcA}) and (\ref{def-Cb}). When it exists, the intersection point $i_{b}$ of $\CC_{b}$ and $\AA$ has been determined in Lemma \ref{inter}. Figure \ref{zeros-nonsym} shows how the zeros of the Faber polynomials distribute, depending on whether or not $\CC_{b}$ and $\AA$ intersect. Figure \ref{arc-nonsym} shows the arc $\AA$ where the zeros accumulate ($\CC_{b}$ and $\AA$ do not intersect in that case). Figure \ref{loop-comp-comp} shows an example when $\CC_{b}$ and $\AA$ intersect.

We consider two cases.

\begin{theorem}\label{complexullman} For $1 <R \cos \theta \leq {3}/{2}$, all zeros of $F_n(z)$ approach $\mathcal A$ as $n\to \infty$ and 
$$\lim_{n\to \infty} \mu_n = (U^{-1})_*(\mu_{[-1,1]}) \quad \hbox{weak-*}.$$

\end{theorem}
\begin{figure}[!tb]
\centering
\includegraphics[width=7.5cm]{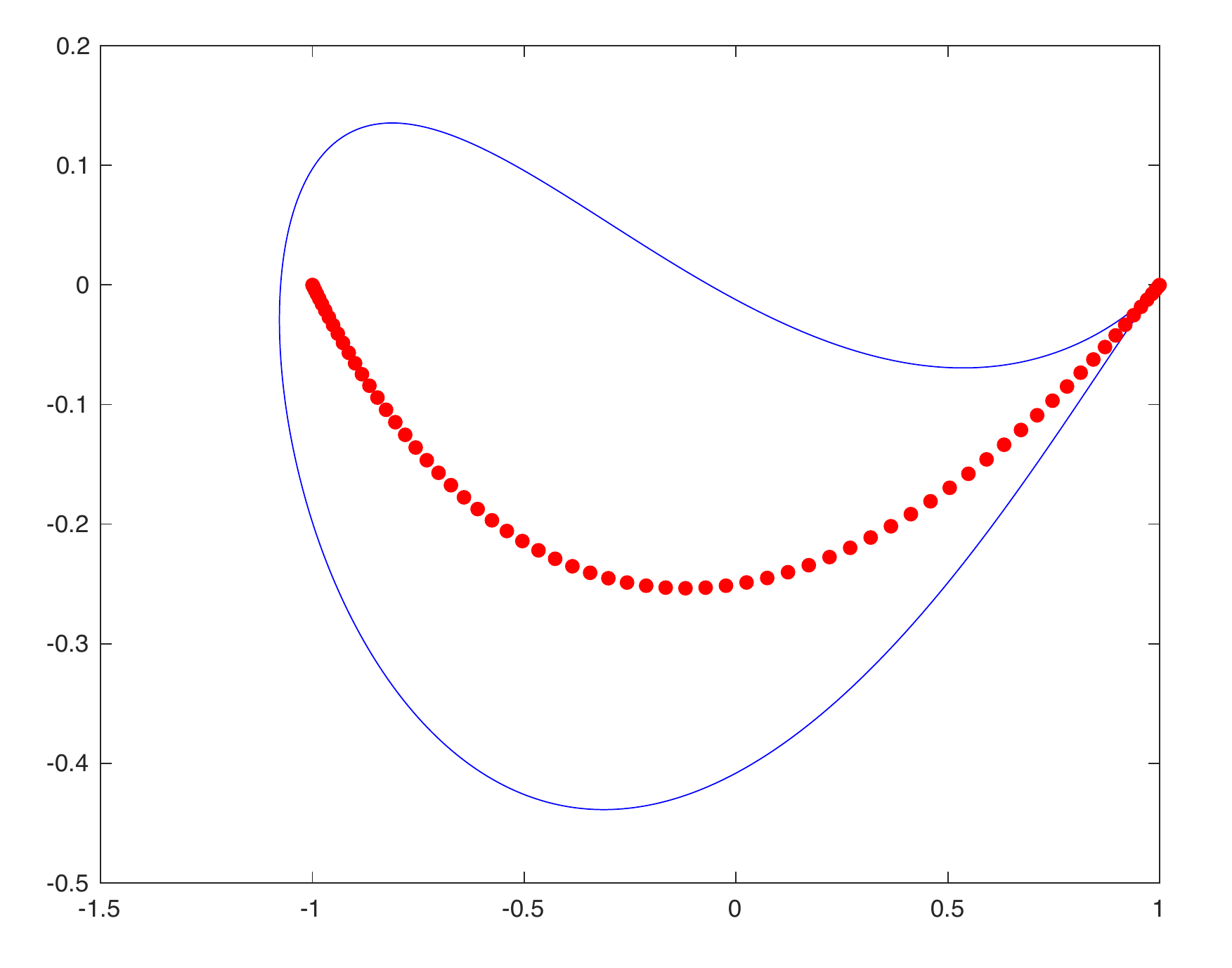}
\includegraphics[width=7.5cm]{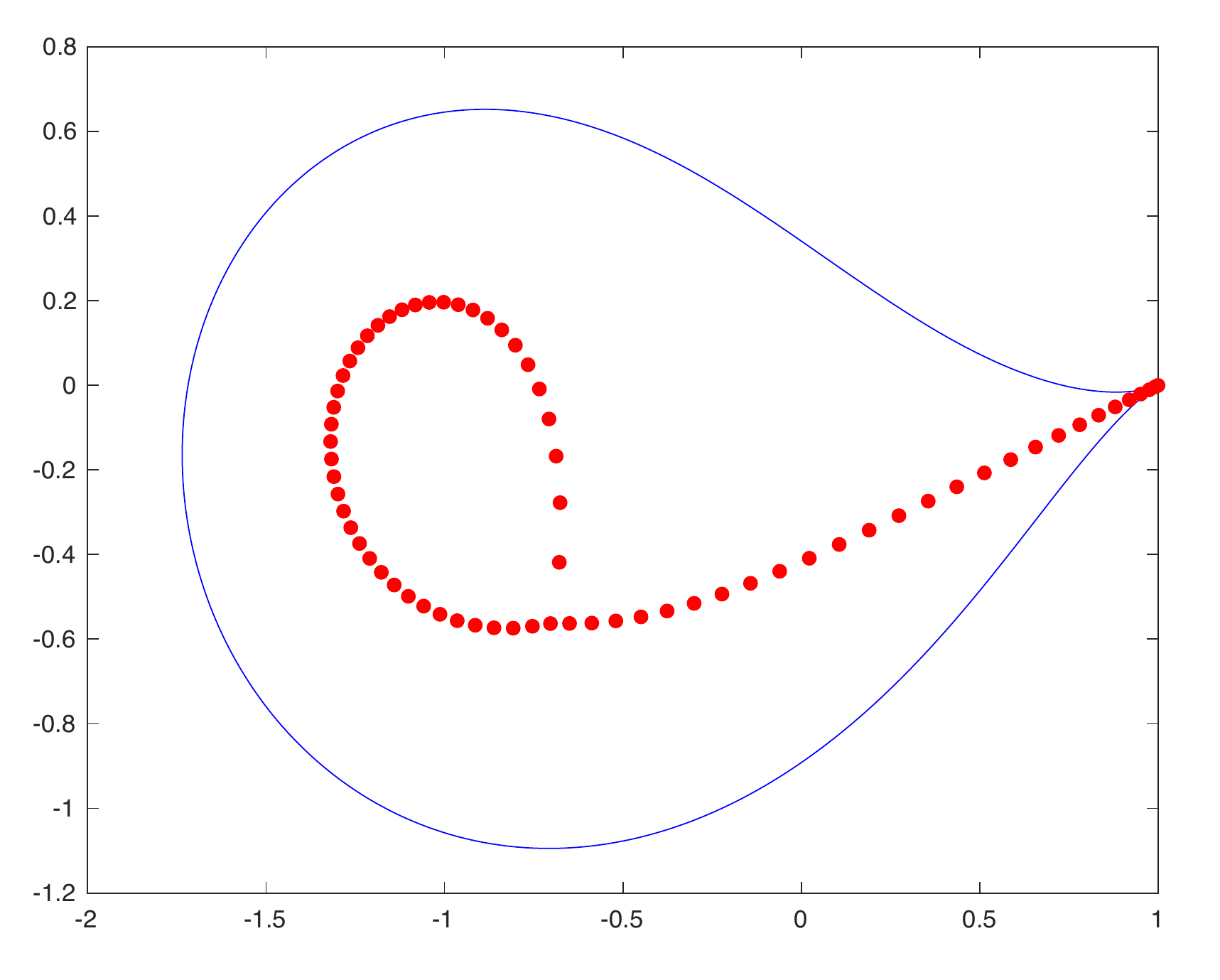}
\caption{Zero distribution of the Faber polynomials $F_{n}(z)$, $n=70$, in the complex case. The parameters for the Joukowski airfoil are $\theta=0.2$, and $R=1.26$ (left), $R=2.1$ (right).}
\label{zeros-nonsym}
\end{figure}
\begin{proof} The difference with Theorem \ref{realullman} is that here the zeros of $F_n$ need not lie on $\mathcal A$ but we first show that they do accumulate there. In this case, by Lemma \ref{inter}, $\mathcal A$ is disjoint from $\mathcal C_b$. Thus we can take a simple, closed contour $\Gamma$ which surrounds $\mathcal A$ and is disjoint from $\mathcal C_b$. If $R \cos \theta = {3}/{2}$ we take $\Gamma$ to contain the point $-1$. We claim that for $z\in \Gamma$, for $n$ sufficiently large, we have the strict inequality
\begin{equation}
\label{strict} 
\left|\left(2T_n\left(\frac{W(z)}{V(z)^{1/2}}\right) -\left(\frac{-b}{V(z)^{1/2}}\right)^n\right)   -2T_n\left(\frac{W(z)}{V(z)^{1/2}}\right)\right|<\left|2T_n\left(\frac{W(z)}{V(z)^{1/2}}\right)\right|.
\end{equation}
This is simply because the left side of (\ref{strict}) is $|({-b}/{V(z)^{1/2}})^n|$ which is at most 1 by Proposition \ref{rhsmall}; while the right side $|2T_n\left({W(z)}/{V(z)^{1/2}}\right)|$ goes to infinity (at a geometric rate) since $z\in \Gamma$ implies $U(z)={W(z)}/{V(z)^{1/2}}\not\in [-1,1]$. The strict inequality continues to hold at $z=-1$ if $R \cos \theta = {3}/{2}$ since the right side is $2=2|T_n(-1)|$ 
while the left side is $1$ since $-1\in \mathcal C_b$ and thus $|b/V(-1)^{1/2}|=1$.
Now both 
$$2T_n\left({W(z)}/{V(z)^{1/2}}\right) -\left({-b}/{V(z)^{1/2}}\right)^n \ \hbox{and} \  2T_n\left({W(z)}/{V(z)^{1/2}}\right)$$ 
are holomorphic functions inside and on $\Gamma$; thus by Rouch\'e's theorem, each has the same number of zeros -- namely $n$ -- inside $\Gamma$. This argument holds for any such $\Gamma$; taking $\Gamma$ closer and closer to $\mathcal A$ shows that all zeros of $F_n(z)$ approach $\mathcal A$ as $n\to \infty$. Indeed, by
choosing a small contour $\gamma$ locally around each zero $\alpha_{k,n}$ of $\zeta \to T_n\left({W(\zeta)}/{V(\zeta)^{1/2}}\right)$ which crosses $\mathcal A$ through the two consecutive extrema of this function around $\alpha_{k,n}$ -- so that (\ref{strict}) holds for all $z\in \gamma$ -- we can apply Rouch\'e's theorem inside $\gamma$. Thus we obtain that the zeros of $\{F_n\}$ asymptotically distribute
like the measure $(U^{-1})_*(\mu_{[-1,1]})$.
\end{proof}

\begin{theorem} For $R \cos \theta > {3}/{2}$, the zeros of $F_n$ accumulate on $\mathcal A_b\cup \LL_b^{+}$ where $\mathcal A_b$ is the portion of the arc $\mathcal A$ from the point $i_b$ to the point 1 and $ \LL_b^{+}$ is the loop in (\ref{lb}) containing the point $i_b$ where it has a corner. Moreover
$$\lim_{n\to \infty} \mu_n = (U^{-1})_*(\mu_{[-1,1]})|_{\mathcal A_b} + (\phi_b^{-1})_*(\eta|_{(c_+,c_-)} )\quad \hbox{weak-*}$$
where $(c_+,c_-)$ is the arc of $\T$ from $c_+$ to $c_-=c_+$ (traversed counterclockwise) where 
$$c_{\pm}:= (\phi_b)_{\pm}(i_b) \in \T.$$
\end{theorem}

\begin{proof} The subarc $\mathcal A_b$ of $\mathcal A$ lies outside of the circle $\mathcal C_b$ so that we can apply a similar Rouch\'e-type argument to conclude that a fixed proportion of the zeros of $F_n$ accumulate on $\mathcal A_b$ and distribute asymptotically like $ (U^{-1})_*(\mu_{[-1,1]})|_{\mathcal A_b}$. For the rest of the zeros, reasoning as in the proof of Theorem \ref{otherreal}, we see that they accumulate on $ \LL_b^{+}$ and distribute asymptotically like $(\phi_b^{-1})_*(\eta|_{(c_+,c_-)} )$.
\end{proof}

{\begin{remark} \label{ES} Recalling from \cite[Theorem III.4.7]{ST} that any weak-* subsequential limit $\mu$ of $\{\mu_n\}$ has a balayage to $\partial K$ which is the equilibrium measure $\mu_K$ of $K$, Theorems 3.1 and 4.1 show that any Joukowski airfoil $K$ with $1<R\cos \theta \leq 3/2$ admits an {\it electrostatic skeleton}; i.e., a positive measure $\mu$ with closed support $S$ in $K$ where $S$ has empty interior and connected complement such that the logarithmic potentials of $\mu$ and $\mu_K$ agree (in our case) on $\C \setminus K$. See \cite{LR} and \cite{SS} for more on this subject.

\end{remark}}
\begin{figure}[!tb]
\centering
\vspace{-1cm}
\includegraphics[width=8cm]{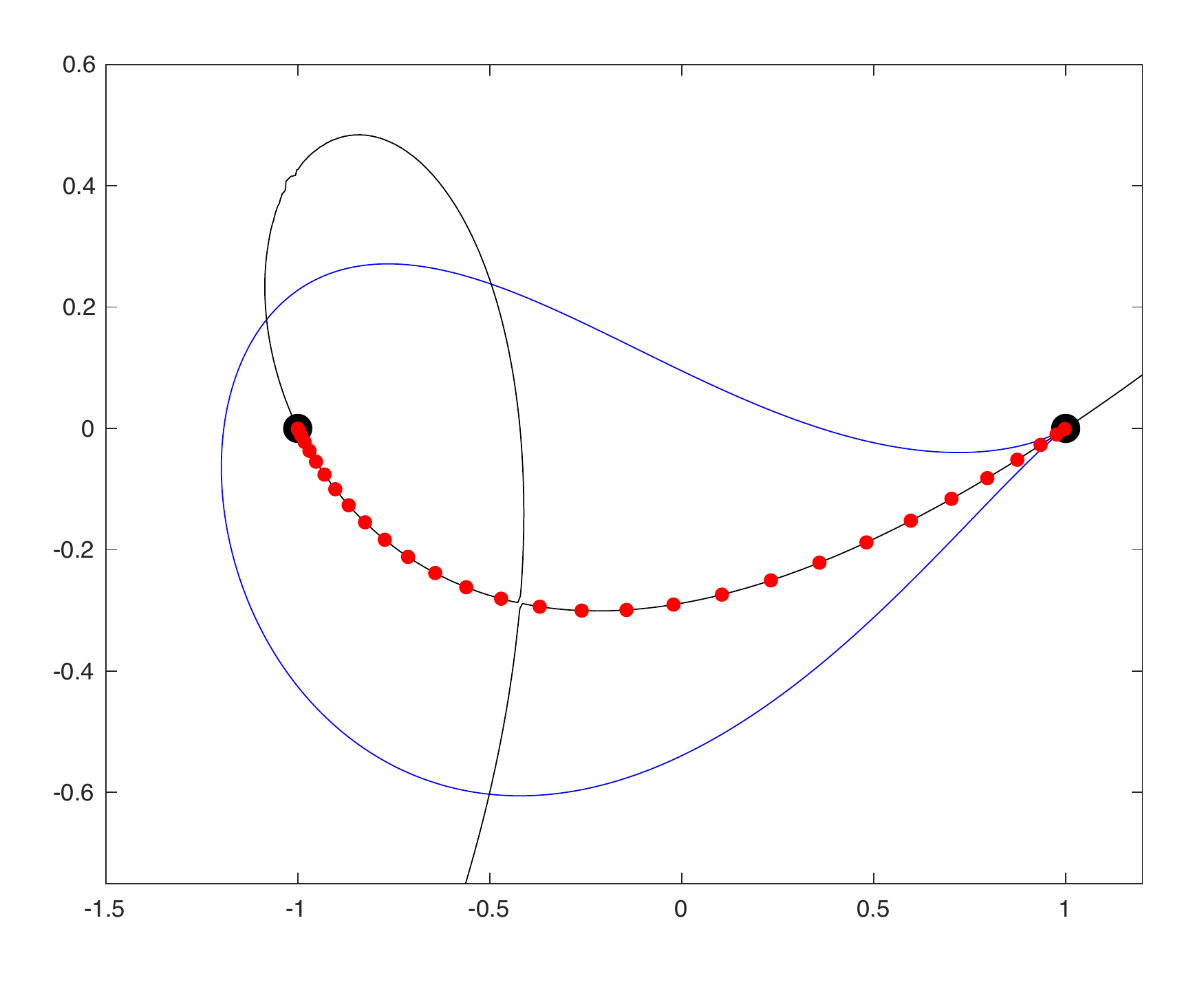}
\caption{Curve (black) where $U^{2}(z)$ is real. Here $\theta=0.2$, and $R=1.45$. The zeros of $F_{n}$ accumulate on the subarc $\AA$ of the curve, between the two dots $1$ and $-1$, where $U(z)\in[-1,1]$.}
\label{arc-nonsym}
\end{figure}

\section{Chebyshev quadrature}\label{quadrat}

There is a connection between Faber polynomials and Chebyshev quadrature. Indeed, let $\mu_K$ denote the equilibrium measure of $K$. Here we are back in the general situation where $K\subset \C$ is a compact set consisting of more than one point with the unbounded component $\Omega$ of $\bar\C\setminus K$ being simply connected. We have the following observation of Kuijlaars (\cite[Lemma 3]{Arno}):

\begin{proposition}\label{arno} Let $n\geq 1$ and let $z_1,...,z_n\in \C$. Then $z_1,...,z_n$ are the zeros of the Faber polynomials $F_n$ associated to $K$ if and only if
$$\int_K z^k d\mu_K(z)= \frac{1}{n}\sum_{j=1}^n z_j^k, \quad k=1,...,n.$$ 
\end{proposition}

This condition says that for any polynomial $p$ of degree at most $n$,
$$\int_K p(z) d\mu_K(z)= \frac{1}{n}\sum_{j=1}^n p(z_j).$$
In other words, $z_1,...,z_n$ are the {\it Chebyshev quadrature nodes of order $n$ for $\mu_K$}.

More generally, given a (say) probability measure $\mu$ with compact support $K\subset \C$, points $z_1,...,z_n\in \C$ are 
{\it Chebyshev quadrature nodes of order $n$ for $\mu$} if for any polynomial $p$ of degree at most $n$,
$$\int_K p(z) d\mu(z)= \frac{1}{n}\sum_{j=1}^n p(z_j)$$
(cf., \cite{Kr}). Proposition \ref{arno} for the interval $[-1,1]$ gives another way to see the Faber polynomials (appropriately normalized) are the classical Chebyshev polynomials of the first kind
$$T_n(z)=\frac{1}{2}\left([z+\sqrt{z^2-1}]^n+ [z-\sqrt{z^2-1}]^n\right).$$
Here recall $d\mu_{[-1,1]}(x) = {1}/({\pi}{\sqrt{1-x^2}})dx$. 
\begin{figure}[!tb]
\centering
\includegraphics[width=9cm]{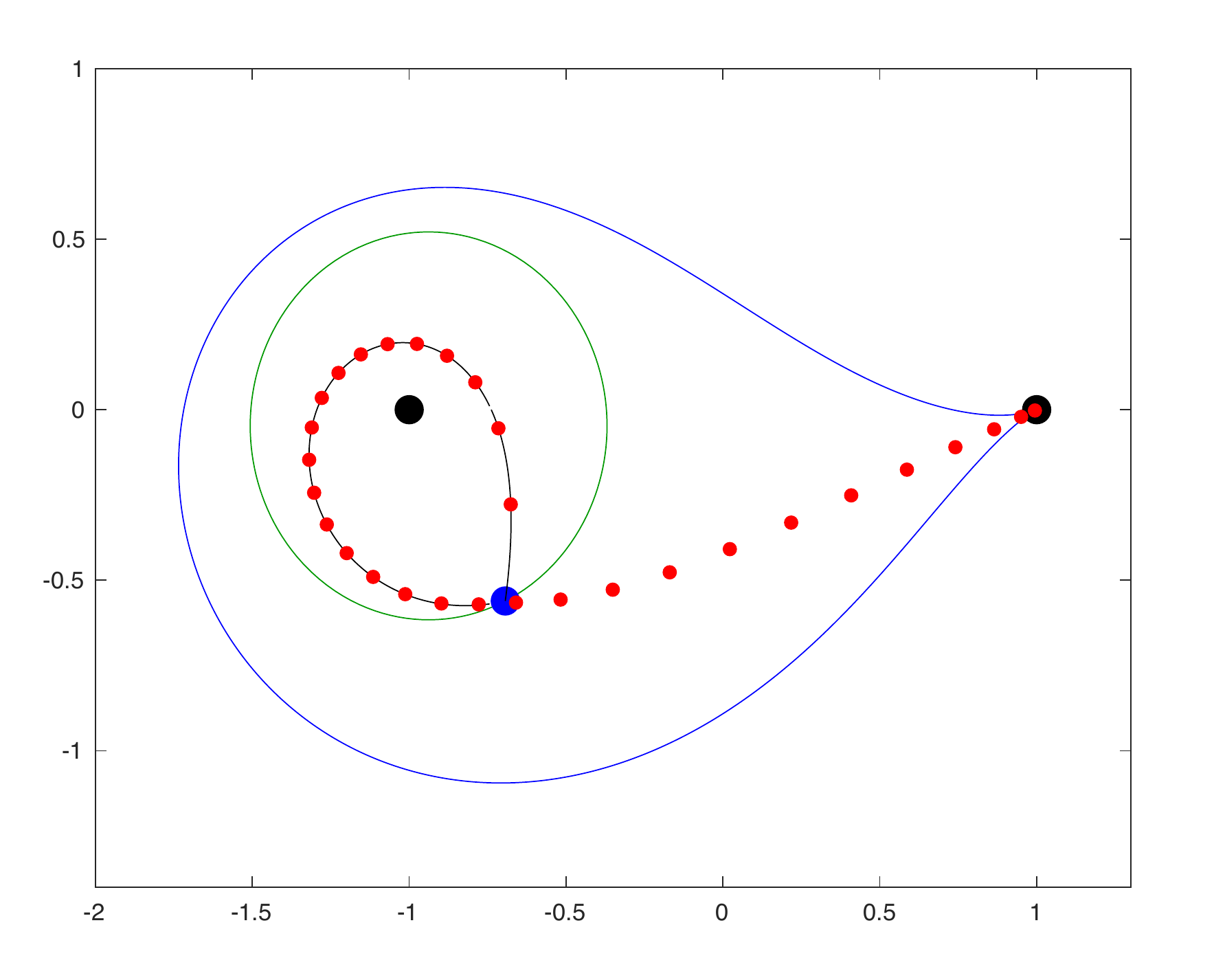}
\caption{Joukowski airfoil (blue) with $\theta=0.2$, $R=2.1$, and the zeros of the Faber polynomials $F_{n}(z)$, $n=30$. The zeros inside the circle $\CC_{b}$ (green) accumulate on the loop (black). The black dots are the real points $1$ and $-1$. The blue dot where the loop ends is the point $\II_{b}$ where the circle $\CC_{b}$ and the arc $\AA$ intersect.}
\label{loop-comp-comp}
\end{figure}
Ullman proved in \cite{U2} that for $-1/4\leq \alpha \leq 1/4$, 
the measure
\begin{equation}\label{ullman} d\mu(x)= \frac{1}{\pi}\frac{1}{\sqrt{1-x^2}}\frac{1+2\alpha x}{1+4\alpha^2+4\alpha x}dx, \quad x\in (-1,1), \end{equation}
supported on $[-1,1]$, 
admits Chebyshev quadrature with nodes $z_1^{(n)},...,z_n^{(n)}$ lying in $[-1,1]$; 
it follows that any weak-* limit of the sequence of measures $\mu_n:= \frac{1}{n}\sum_{j=1}^n \delta_{z_j^{(n)}}$ has the same moments as those of $\mu$, hence is equal to $\mu$, and thus the whole sequence $\mu_{n}$ converges weak-* to $\mu$. Indeed, he shows that 
$$\left( z+\sqrt{z^2-1}+2\alpha\right)^n+\left( z-\sqrt{z^2-1}+2\alpha\right)^n -(2\alpha)^n$$
is a polynomial of degree $n$ with zeroes at $z_1^{(n)},...,z_n^{(n)}$. This is a special case of our formula (\ref{use?}). Hence Ullman's Chebyshev quadrature nodes for the measure $\mu$ in (\ref{ullman}) are precisely the zeros of the Faber polynomials corresponding to the situation of Theorem \ref{realullman}. Here $\alpha=({R-1})/{2}$. Since $R>1$, the condition $(R-1)/2\leq 1/4$ becomes $1< R \leq 3/2$ as in our theorem. 
Note also that the measure $\mu$ in (\ref{ullman}) corresponds to the limit measure in (\ref{push-expl}) (and the balayage of $\mu$ to $\partial K$, where $K=K(R)$ is the corresponding Joukowski airfoil, is $\mu_K$).

Although it is not clear to us how Ullman arrived at his family of measures in (\ref{ullman}), we make the following observation. Suppose that a compact set $K$ is given with the property that for each $n\geq 1$, the zeros 
$z_1^{(n)},...,z_n^{(n)}$ of the Faber polynomial $F_{n}$ for $K$ lie in some interval $[a,b]\subset\R$, and moreover that the corresponding counting measures $\mu_{n}$
converge weak-* to a measure $\mu$. It then follows that $\mu$ admits Chebyshev quadrature with nodes $z_1^{(n)},...,z_n^{(n)}$, $n\geq1$, lying in $[a,b]$.

\vspace{2cm}
{\obeylines
\texttt{N. Levenberg, nlevenbe@indiana.edu
Indiana University, Bloomington, IN 47405 USA
\medskip
F. Wielonsky, franck.wielonsky@univ-amu.fr
Universit\'e Aix-Marseille, CMI 39 Rue Joliot Curie
F-13453 Marseille Cedex 20, FRANCE }
}


\begin{thebibliography}{99}

\bibitem{HS} M. X. He and E. B. Saff, The zeros of Faber polynomials for an $m-$cusped hypocycloid, \emph{ J. Approx. Theory}, \textbf{78}, no. 3, (1994), 410-432. 

\bibitem{Kr} V. I. Krylov, \emph{Approximate calculation of integrals}, 
English translation by Arthur H. Stroud, The Macmillan Co., New York-London, 1962.

\bibitem{Arno} A. Kuijlaars, The zeros of Faber polynomials generated by an $m-$star, \emph{Mathematics of Computation}, \textbf{65}, no. 113, (1996), 151-156.

\bibitem{KS} A. Kuijlaars and E. Saff, Asymptotic distribution of the zeros of Faber polynomials, \emph{Math. Proc. Camb. Phil. Soc.}, \textbf{118}, (1995), 437-447.

\bibitem{L} J. Liesen, Faber polynomials corresponding to rational exterior mapping functions, \emph{Constr. Approx.}, \textbf{17}, (2001), 267-274.

{\bibitem{LR} E. Lundberg and K. Ramachandran, Electrostatic Skeletons, \emph{Ann. Acad. Scient. Fenn. Math.},  \textbf{40}, (2015), 397-401.}

\bibitem{MD} E. Mina-Diaz, On the asymptotic behavior of Faber polynomials for domains with piecewise analytic boundary,  \emph{Constr. Approx.}, \textbf{29}, (2009), 421-448.

\bibitem{SS} E.B. Saff, N. Stylianopoulos, On the zeros of asymptotically extremal polynomial sequences in the plane, \emph{J. Approx. Theory},  \textbf{191}, (2015) 118-127.

\bibitem {ST} E. Saff, V. Totik, \emph{Logarithmic potentials with external fields}, Springer-Verlag, Berlin, 1997.

\bibitem{U1} J. Ullman, Studies in Faber polynomials I, \emph{Trans. A.M.S.}, \textbf{94}, (1960), 515-528.

\bibitem{U2} J. Ullman, A class of weight functions that admit Tchebycheff quadrature, \emph{Michigan Math. J.}, \textbf{13},  (1966), 417-423. 

\end{thebibliography}
\end{document}